\documentclass[11pt]{amsart}
\usepackage{amsmath, fullpage}
\usepackage{amsxtra}
\usepackage{amscd}
\usepackage{amsthm}
\usepackage{amsfonts}
\usepackage{amssymb}
\usepackage{eucal}
\usepackage{mathrsfs}
\usepackage{color, graphics}
\usepackage[all]{xy}
\usepackage{hyperref}
\usepackage{pgf,tikz}
\usetikzlibrary{arrows}

\theoremstyle{plain} 
\newcommand \reg {\mathrm {reg}}

\newcommand \Gin {\ensuremath{\mathrm{Gin}}}
\newcommand \In {\ensuremath{\mathrm{in}}}
\newcommand \Proj {\ensuremath{\mathrm{Proj}}}
\newcommand \red {\ensuremath{\mathrm{red}}}
\newcommand \mult {\ensuremath{\mathrm{mult}}}
\newcommand \p {\ensuremath{\mathbb{P}}}
\newcommand \lex {\ensuremath{\mathrm{GLex}}}
\newcommand \Tan {\ensuremath{\mathrm{Tan}}}

\theoremstyle{definition}
\newtheorem{Thm}{Theorem}[section]
\newtheorem{Defn}[Thm]{Definition}
\newtheorem{lem}[Thm]{Lemma}
\newtheorem{remk}[Thm]{Remark}

\newtheorem{Ex}[Thm]{Example}
\newtheorem{prop}[Thm]{Proposition}
\newtheorem{Cor}[Thm]{Corollary}

\begin{document}

\title[Generic Initial ideals of Singular Curves in Graded Lex Order.]{Generic Initial ideals of Singular Curves in Graded Lexicographic Order.}
\author[J.\ Ahn$^{1}$, S.\ Kwak$^{2}$ and Y.\ Song$^{*}$]{Jeaman Ahn${}^{1}$ , Sijong Kwak${}^{2}$ and YeongSeok Song${}^{*}$}
\address{ Mathematics Education, Kongju National University,
182 Sinkwan-dong Kongju-si Chungnam, 314-702, Korea}
\email{jeamanahn@kongju.ac.kr}
\thanks{${}^{1}$ The first author was supported in part by the National Research Foundation of Korea funded
by the Ministry of Education, Science and Technology(grant No. 2010-0025762)}
\address{Department of Mathematics, Korea Advanced Institute of Science and Technology,
373-1 Gusung-dong, Yusung-Gu, Daejeon, Korea}
\email{skwak@kaist.ac.kr}
\thanks{${}^{2}$ The second author was supported in part by the National Research Foundation of Korea funded
by the Ministry of Education, Science and Technology(grant No. 2010-0001652)}
\address{Department of Mathematics, Sogang University,
1 Sinsu-dong, Mapo-gu, Seoul, 121-742, Korea}
\email{lassei@kaist.ac.kr}
\thanks{${}^{*}$ Corresponding author.}

\begin{abstract}
In this paper, we are interested in the generic initial ideals of \textit{singular} projective curves
with respect to the graded lexicographic order. Let $C$ be a \textit{singular} irreducible projective curve of degree $d\geq 5$ with the arithmetic genus  $\rho_a(C)$ in $\p^r$ where $r\ge 3$. If $M(I_C)$ is the regularity of the lexicographic generic initial ideal of $I_C$ in a polynomial ring $k[x_0, \ldots, x_r]$ then we prove that $M(I_C)$ is $1+\binom{d-1}{2}-\rho_a(C)$ which is obtained
from the monomial
$$
x_{r-3}  x_{r-1}\,^{\binom{d-1}{2}-\rho_a(C)},
$$
provided that $\dim\Tan_p(C)=2$ for every singular point $p \in C$.
This number is equal to one plus the number of non-isomorphic points under a generic projection
of $C$ into $\p^2$.
Our result generalizes the work of J. Ahn \cite{A} for \textit{smooth} projective curves and  that of
A. Conca and J. Sidman \cite{CS} for \textit{smooth} complete intersection curves in $\p^3$.
The case of singular curves was motivated by \cite[Example 4.3]{CS} due to A. Conca and J. Sidman.
We also provide some illuminating examples of our results via calculations done with {\it Macaulay 2}
and \texttt {Singular} ~\cite{DGPS, GS}.
\end{abstract}

\maketitle
\tableofcontents \setcounter{page}{1}

\section{Introduction}\label{section_1}
Let $R=k[x_0, \ldots, x_r]$ be a polynomial ring over an algebraically closed field $k$ of characteristic zero and $I$ be a homogeneous ideal of $R$. If $X$ is a non-degenerate reduced closed subscheme in $\p^r$ we write
$I_X$ for the saturated defining ideal of $X$ in the polynomial ring $R$.

Bayer and Mumford in \cite{BM} introduced the {\it regularity} of the initial ideal of $I$ with respect to a term order $\tau$ as a measure of the complexity of computing Gr\"{o}bner bases. Even though this depends on the choice of coordinates, it is constant in generic coordinates by the result of Galligo \cite{Ga}. He has proved that the initial ideals of $I$ in generic coordinates are invariant, which is the so-called generic initial ideal of $I$ with respect to $\tau$, denoted by $\Gin_{\tau}(I)$. In characteristic zero, it was shown in \cite{BS} that the regularity of $\Gin_{\tau}(I)$ is exactly the maximum of the degrees of its minimal generators.

One of the important problems is to bound the regularity of the generic initial ideal of $I$ for a given term order $\tau$ on monomials. Two of the most commonly used orderings are the graded lexicographic ordering, and the graded reverse lexicographic ordering. Many people have studied generic initial ideals with respect to the reverse lexicographic term ordering, as these ideals have essentially
best-case complexity due to a result of Bayer and Stillman (for examples, \cite{BM, BS, BEL, C, Giaimo, GLP, K, K1, Laz, Pin, PS, Ran, SV}).
However, much less is known about the generic initial ideals with respect to the graded lexicographic term ordering. One expects them to require
many more generators than the reverse lexicographic initial ideals, but their precise behavior has been very little known (\cite{A, AKS, CS}).

In this paper, we continue the study of the lexicographical generic initial ideals of \textit{singular} projective curves. Our main result gives a relationship between the complexity of algebraic computations with the ideal of a singular curve and the geometry of its generic projection to the plane. It states that if $C$ is a \textit{singular} irreducible projective curve of degree $d\geq 5$ with the arithmetic genus  $\rho_a(C)$ in $\p^r$ where $r\ge 3$ then the regularity of the lexicographic generic initial ideal of a singular curve $C$ in projective space is precisely $1+\binom{d-1}{2}-\rho_a(C)$, which is one plus the number of non-isomorphic points under a generic projection
of $C$ into $\p^2$, provided that $\dim\Tan_p(C)=2$ for every singular point $p\in C$.
Moreover it turns out that the regularity is obtained from the monomial generator $ x_{r-3}  x_{r-1}\,^{\binom{d-1}{2}-\rho_a(C)}$ of $\Gin(I_C)$.

We uses M. Green's partial elimination ideals and careful work with their Hilbert functions to achieve the result, which previously has been used in \cite{A}. Main ideas employed in this paper are to reduce the problem to the case of singular curves in $\p^3$ and to show that the first partial elimination ideal of $I_C\subset K[x_0,x_1,x_2,x_3]$ is a radical ideal in generic coordinates, under the assumption that
$\dim\Tan_p(C)=2$ for every singular point $p \in C$. In process of the proof, this ideal turns out to be the defining ideal of the set of non-isomorphic points under a generic projection of $C$ into $\p^2$.

Our result generalizes the works of J. Ahn \cite{A} and A. Conca and J. Sidman \cite{CS} who proved the same formula for the case of \textit{smooth} projective curves and for \textit{smooth} complete intersection curves in $\p^3$, respectively.

Finally, we remark that our result is not true if $\dim \Tan_p (C) > 2$. The example of A. Conca and J. Sidman \cite[Example 4.3]{CS} is a complete intersection curve $C$ defined by
$x^3-yz^2$ and $y^3-z^2t$ with one singular point $p=[0,0,0,1]$. One can compute $\dim \mathrm{Tan}_p(C)=3$ and $\delta_p = 10$ with \texttt {Singular} \cite{DGPS}. In this case the regularity of the lexicographic generic initial ideal of $I_C$ is $16$, which is not $1+\binom{9 -1}{2}-\rho_a(C)=19$ (see Example 3.6 for the details).
\bigskip

{\bf Acknowledgements} We are very grateful to the anonymous referee for valuable and helpful suggestions.
In addition, {\it Macaulay 2} and \texttt {Singular} have been useful to us in computations of generic initial ideals
of partial elimination ideals and the delta invariant.


\section{Notations and known facts}\label{section_2}

\begin{itemize}
\item[(a)] We work over an algebraically closed field $k$ of characteristic zero.

\item[(b)] For a homogeneous ideal $I$, the Hilbert function of $R/I$ is defined by
$H(R/I, m):=\dim_k (R/I)_m$ for any non-negative integer $m$.
We denote its corresponding Hilbert polynomial by $P_{R/I}(z)\in \Bbb Q[z]$. If $I=I_X$
then we simply write $P_X(z)$ instead of $P_{R/I_X}(z)$.

\item[(c)] Given a homogeneous ideal $I \subset R$ and a term order $\tau$, there is a Zariski open subset
$U \subset GL_{r+1}(k)$ such that $\In_{\tau}(g(I))$ for $g \in U$ is constant. We will call $\In_{\tau}(g(I))$
 the generic initial ideal of $I$ for $g \in U$ and denote it by $\Gin_{\tau}(I)$. One can say that $I$ is in generic coordinates
 if $\In_{\tau}(I)=\Gin_{\tau}(I)$.

\item[(d)] For a homogeneous ideal $I\subset R$, let $M(I)$ denote the maximum of the degrees of minimal generators of $\Gin_{\mathrm{GLex}}(I)$.

\item[(e)] For a homogeneous ideal $I\subset R$, consider a minimal free resolution
$$\cdots \rightarrow \bigoplus_jR(-i-j)^{\beta_{i,j}(I)}\rightarrow\cdots\rightarrow\bigoplus_jR(-j)^{\beta_{0,j}(I)}\rightarrow I\rightarrow 0$$
of $I$ as a graded $R$-modules.
We say that $I$ is $m$-regular if $\beta_{i,j}(I)=0$ for all $i\geq 0$ and $j\geq m$.
The Castelnuovo-Mumford regularity of $I$ is defined by
\[\reg(I):=\min\{\,m\,\mid\, I \textup{ is } m \textup{-regular}\}.\]

\item[(f)] If $I$ is a Borel fixed monomial ideal then $\reg(I)$ is exactly the maximal degree of minimal generators of $I$ (see \cite{BS},\cite {G}). This implies that  $M(I)=\reg(\Gin_{\mathrm{GLex}}(I))$.

\item[(g)] Let $C$ be an integral projective scheme of dimension $1$ over $k$, and $f: \widetilde{C} \longrightarrow C$ be its normalization. We write $\delta_p$ for the length of $(f_{*} \mathcal{O}_{\widetilde C})_p / \mathcal{O}_{C,p}$ as an $\mathcal{O}_{C,p}$-module for each $p \in C$.
     Note that if a singular point $p$ is a node or an ordinary cusp then $\delta_p = 1$.\cite[Exercise IV 1.8(c)]{H}
\end{itemize}
\bigskip

We recall some definitions and known facts which will be used
throughout the remaining parts of the paper. Unless otherwise stated, we always assume the graded lexicographic term ordering.

\begin{Thm}\cite[Theorem 1.2]{A}
Let $X$ be an integral scheme in $\p^r$ and let $\pi$ be a generic projection of $X$ to $\p^{r-1}$.
Suppose that $\pi$ is an isomorphism. Then $M(I_X)=M(I_{\pi(X)})$.
\end{Thm}

\begin{Defn}\cite{CS, G}\label{defpartial}
Let $I$ be a homogeneous ideal in $R=k[x_0, \ldots, x_r]$. If $f \in I_d$ has leading term $\In(f)=x_{0}^{d_{0}} \cdots x_{r}^{d_{r}}$,
we will set $d_{0}(f)=d_0$, the leading power of $x_0$ in $f$. We let
\begin{equation*}
\widetilde{K}_{i}(I)= \bigoplus_{d \geq 0} \{f \in I_{d} \mid  d_{0}(f) \leq i  \}.
\end{equation*}
If $f \in \widetilde{K}_{i}(I)$, we may write uniquely $f=x_{0}^{i}\overline{f} + g$,
where $d_{0}(g) < i$. Now we define $K_{i}(I)$ as the image of $\widetilde{K}_{i}(I)$ in $\bar{R}=k[x_1 \ldots x_r]$
under the map $f\rightarrow \overline{f}$ and we call $K_{i}(I)$ the $i$-th partial elimination ideal of $I$ .
\end{Defn}

\begin{remk}
We have an inclusion of the partial elimination ideals of $I$:
$$ I\cap\bar{R}= K_0(I)\subset K_1(I)\subset \cdots \subset K_i(I)\subset  K_{i+1}(I)\subset\cdots \subset \bar{R}=k[x_1 \ldots x_r].$$
Note that if $I$ is in generic coordinates and $i_0=\min\{i \mid I_i \neq 0 \}$ then $K_{i}(I)=\bar{R}$
for all $i\geq i_0$.
\end{remk}

The following result gives the precise relationship between partial elimination ideals and the geometry of the projection map from $\p^r$
to $\p^{r-1}$. For a proof of this proposition, see \cite[Propostion 6.2]{G}.

\begin{prop}\label{defset}
Let $X \subset \p^{r}$ be a reduced closed subscheme and let $I_X$ be the defining ideal of $X$.
Suppose $p= [1,0, \ldots ,0 ] \in \p^r \setminus X$ and that $\pi : X \rightarrow \p^{r-1}$ is the projection from
the point $p \in \p^r$ to $x_{0} = 0$. Then, the radical ideal $\sqrt{K_{i}(I_{X})}$ defines the algebraic set
$\{ q \in \pi(X) \mid \mult_{q}(\pi(X))> i \}$ set-theoretically.
\end{prop}

Thus, we can define the following two projective schemes associated with the partial elimination ideals:
\begin{equation*}
Y_{i}(X):=\Proj(\bar{R}/\sqrt{K_{i}(I_{X}))}\subset Z_{i}(X):=\Proj(\bar{R}/K_{i}(I_{X})).
\end{equation*}
It is clear that $Z_{i}(X)_{\red}=Y_i(X)$ and if ${K_{i}(I_{X})}$ is reduced, then $Y_{i}(X)=Z_i(X)$.
\bigskip

It is natural to ask what is a Gr\"{o}bner basis of $K_i(I)$? Recall that any non-zero polyomial $f$ in $R$
can be uniquely written as $f=x^t\bar f+g$ where $d_0(g)<t$. A. Conca and J. Sidman \cite{CS} show that
if $G$ is a Gr\"{o}bner basis for an ideal $I$ then the set
$$
G_i=\{ \bar{f} \mid f\in G \mbox{ with } d_0(f)\leq i\}
$$
is a Gr\"{o}bner basis for $K_i(I)$. However if $I$ is in generic coordinates then there is a more refined
Gr\"{o}bner basis for $K_i(I)$, which plays an important role in this paper.
For lack of reference, we give a proof of the following Proposition.
\begin{prop}\label{thm2}
Let $I$ be a homogeneous ideal {\it in generic coordinates} and $G$ be a Gr\"{o}bner basis for $I$ with respect to the graded lexicographic order.
Then, for each $i\geq 0$,
\begin{itemize}
\item[(a)] the $i$-th partial elimination ideal $K_i(I)$ is in generic coordinates;
\item[(b)] $G_i = \{ \bar{f} \mid f \in G ~\mbox{with} ~d_0 (f) = i\}$
is a Gr\"{o}bner basis for $K_i(I)$.
\end{itemize}
\end{prop}
\begin{proof}
(a) is in fact proved in Proposition~3.3 in \cite{CS}.
For a proof of (b),  it suffices to show that $\langle \In(G_i)\rangle = \In(K_i (I))$
by the definition of Gr\"obner bases. Since $G_i \subset K_i (I)$,
we only need to show that $\langle \In(G_i) \rangle \supset \In(K_i (I))$.
Now, we denote $\mathcal{G}(I)$ by the set of minimal generators of $I$.
Let $m \in \In(K_i (I))$ be a monomial. Then there is a monomial generator $M \in \mathcal{G}(\In(K_i (I)))$
such that $M$ divide $m$.

We claim that $x_0^i M \in \mathcal{G}(\In(I))$ if and only if $ M \in \mathcal{G}(\In(K_i (I))).$

If the claim is proved then we will be done. Indeed, for $ M \in \mathcal{G}(\In(K_i (I)))$,
we see that $x_0^iM \in \mathcal{G}(\In(I))$. This implies that there exists a polynomial
$f = x_0^i \bar{f} +g \in G$  with $d_0(g) < i$ such that
$$\In(f)=x_0^i\In(\bar f)=x_0^i M.$$
This means that $M=\In(\bar f) \in \langle \In(G_i) \rangle$. Thus we have $m \in \langle \In(G_i) \rangle$.\\
Here is a proof of the claim: suppose that $x_0^i M \in \mathcal{G}(\In(I))$ then we can say that $x_0^i M \in \In(I)$.
Thus there is a polynomial $f =x_0^i \bar{f} +g \in I$ such that $d_0(g) < i$ and $\In(f) =x_0^i \In(\bar f) = x_0^i M$.
By the definition of partial elimination ideals, we have that $\bar{f} \in K_i(I)$, which means $M \in \In(K_i(I))$.
Assume that $M \notin \mathcal{G}(\In(K_i(I)))$. Then for some monomial $N \in \mathcal{G}(\In(K_i(I)))$ such that
$N$ divide $M$.
This implies that
$$
x_0^i N \in \In(I) ~\mbox{and}~ x_0^i N \mid x_0^i M,
$$
which contradicts the fact that $x_0^i M$ is a minimal generator of $\In(I)$.
Thus $M$ is contained in $\mathcal{G}(\In(K_i(I)))$.

Conversely, suppose that there is $M \in \mathcal{G}(\In(K_i(I)))$ such that $x_0^i M \notin \mathcal{G}(\In(I))$.
Then we may choose a monomial $x_0^j N \in \mathcal{G}(\In(I))$ satisfying
\begin{equation}\label{eq:000001}
x_0 \nmid N ~\mbox{and}~x_0^j N \mid x_0^i M.
\end{equation}
Note that (\ref{eq:000001}) implies that $i \geq j \geq 0$.
Since $N \in \In(K_j(I))$ and $K_0(I) \subset K_1(I) \subset \cdots$, it is obvious that
$N \in \In(K_i(I))$ and $N$ divides $M$.
Now, we claim that $N$ can be chosen to be different from $M$. If $N = M$ then $j$ must be less than $i$.
Denote $N$ by $x_1^{j_1} \cdots x_r^{j_r}$ and choose $j_t\neq 0$. By (a), note that $K_{i}(I)$ is
in generic coordinates and so we may assume that $\In(K_i(I))$ has the Borel-fixed property.
Therefore, if we set
$N^{'}=N/x_{j_t}$ then $x_0^{j+1} N^{'} \in \In(I)$.
Replace $x_0^j N$ by $N^{''} =x_0^{j+1} N^{'}$. Then $N^{'} \in \In(K_{j+1}(I))$.
Since $j+1\leq i$, we can say that $N^{'} \in \In(K_i(I))$ and $N^{'}$ divides $M$ with $N^{'} \neq M$.
This contradicts the assumption that $M \in \mathcal{G}(\In(K_i(I)))$.
\end{proof}

We have the following immediate Corollary from Proposition~\ref{thm2}.
\begin{Cor}\label{cor:305}
For a homogeneous ideal $I\subset R=k[x_0,\ldots, x_r]$ {\it in generic coordinates}, we have
\[M(I)=\max \{ M(K_i(I))+i~ \mid~ 0\leq i\leq \beta\},\]
where $\beta=\min\{j ~\mid ~I_{j}\neq 0 \}$.
\end{Cor}


\section{Generic Initial Ideals of Singular Curves.}\label{section_3}
As mentioned in the introduction, $M(I_{C})$ can be computed precisely in terms of degree and genus
for a {\it smooth} integral curve $C$ in $\p^{r}$, $r \geq 3$.
In this section, we generalize the results for smooth curves in \cite{A} to non-degenerate {\it singular} curves in
$\p^{r}$, $r \geq 3$. We are motivated by \cite[Example 4.3]{CS} due to A. Conca and J. Sidman.
\begin{remk}\label{remk}
We will use the following well known facts to prove our main results.
\begin{enumerate}
\item[(a)](Trisecant Lemma)\label{trieq} Let $C$ be a reduced, irreducible curve in $\p^r$ where $r \geq 3$.
There are at most 1-dimensional trisecant lines to $C$,
which is equivalent to the assertion that not every pair of points of $C$ lie on a trisecant line (see ~\cite{ACGH}).

\item[(b)] Let $C$ be an integral curve in $\p^r$, $r \geq 3$, and $\dim \mathrm{Tan_p (C)}=2$ for any $p \in \mathrm{Sing(C)}$.
Then we can choose a generic point $q \notin \mathrm{Tan_p(C)}$ such that $\pi_q : C \longrightarrow \p^{r-1}$ is an isomorphic projection.
Furthermore, $M(I_C)=M(I_{\pi_q (C)})$.
\end{enumerate}
\end{remk}

From now on, we consider the Hilbert functions of two subschemes $Y_i(C)\subset Z_i(C) \subset \p^2$ associated to the partial elimination ideals
$K_i (I_C)$, $i=0,1$ for a singular projective curve $C$.

\begin{lem}\label{lem3.1}
Let $I_C \subset k[x_0, \ldots ,x_3]$ be a defining ideal of an integral, possibly singular, curve $C$ in $\p^3$.
Then $\deg(\bar{R}/K_1 (I_C))=\binom{d-1}{2}-\rho_a(C)$.
\end{lem}
\begin{proof}
The Hilbert function of $I_C$ is decomposed by the partial elimination ideals $K_i(I_C)$ as follows;
\begin{equation}\label{eq:HF}
H(R/I_C , m) = \sum_{i=0}^{\infty} H(\bar R /K_i (C), m-i).
\end{equation}
This comes from the following combinatorial identity
$$
\binom{m+d}{d} = \sum_{i=0}^{d} \binom{m-1+d-i}{d-i}.
$$
By Remark~\ref{remk}(a), we know that there is no trisecant line to $C$ passing through a general point.
This means that the zero locus of $K_i(I_C)$ is empty for $i \geq 2$ by Proposition~\ref{defset}.
So, $H(\bar R /K_i (C), m)=0$ for $m \gg 0$ and $i\ge 2$. Thus, the equality~(\ref{eq:HF}) can be reformulated by
\begin{equation}\label{eq:HF1}
P_C (m) = P_{\pi(C)}(m)+P_{Z_1(C)}(m-1) ~\mbox{for $m \gg 0$}.
\end{equation}
Since $\pi(C)$ is a plane curve of degree $d=\deg(C)$ and arithmetic genus $\rho_{a}(\pi(C))=\binom{d-1}{2}$, we know that
$P_C (m)=dm + 1-\rho_a(C)$, and  $P_{\pi(C)}(m)=dm +1-\binom{d-1}{2}$.
Consequently,
$$
\deg(\bar{R}/K_1 (I_C))=P_C (m)-P_{\pi(C)}(m)=\binom{d-1}{2}-\rho_a(C).
$$
\end{proof}

\begin{Thm}\label{rdeg}
Let $C$ be a non-degenerate integral curve of degree $d$ and
arithmetic genus $\rho_{a}(C)$ in $\p^{3}$. Assume that $\dim \Tan_p(C)=2$ for every singular point $p\in C$.
Then $K_1(I_C)$ is a radical ideal defining a set of reduced points $Y_1(C)$ of degree $\binom{d-1}{2}-p_a(C)$, which is the number of non-isomorphic points under a generic projection of $C$ into $\p^2$.
\end{Thm}

\begin{proof}
Let $f:\widetilde{C} \longrightarrow C$ be the normalization of $C$. Then we have the following exact sequence
\begin{equation*}\label{1}
0 \rightarrow \mathcal{O}_{C} \rightarrow f_{*}\mathcal{O}_{\widetilde{C}} \rightarrow \sum_{p \in C} (f_{*} \mathcal{O}_{\widetilde C})_p/\mathcal{O}_{C,p} \rightarrow 0
\end{equation*}
where $(f_{*} \mathcal{O}_{\widetilde C})_p$ is the integral closure of $\mathcal{O}_{C,p}$.
Thus we have the equation
\begin{equation}\label{eq:300}
\begin{array}{llllllllllllllllll}
\sum_{p \in C} \delta_{p}&= \chi (f_{*} \mathcal{O}_{\widetilde{C}})-\chi (\mathcal{O}_C) \\
                           &=(1-\rho_a(\widetilde C) ) - (1-\rho_a(C)) \\
                           &=\rho_{a} (C) - \rho_{a}(\widetilde{C})
\end{array}
\end{equation}
where $\delta_p = \mathrm{length}((f_{*} \mathcal{O}_{\widetilde C})_p/\mathcal{O}_{C,p})$.
Now consider the following commutative diagram:
\begin{equation*}
\xymatrix{
\widetilde C \ar[r]^f \ar[dr]_{\pi ^{'}} & **[r] C \ar[d]^{\pi}\subset \p^3\\
                                         & **[r] \pi(C) \subset \p^2  }
\end{equation*}
where $\pi^{'} =\pi \circ f : \widetilde{C} \longrightarrow \p^2$. The assumption that $\dim \Tan_p(C)=2$ for every singular point $p\in C$ implies that the generic projection $\pi : C \longrightarrow \p^2$ gives a local isomorphism around every singular point $p\in C$ and thus we have
\[
\begin{array}{llllllllllllll}
\delta_p &= \mathrm{length}((f_{*} \mathcal{O}_{\widetilde C})_p/\mathcal{O}_{C,p})\\
            &= \mathrm{length}((\pi^{'}_{*}\mathcal{O}_{\widetilde{C}})_q/\mathcal{O}_{\pi(C),q} )&=\delta_q
\end{array}
\]
where $q=\pi(p)$. By virtue of Remark~\ref{remk}, we see that the fiber of a generic projection of the curve $C$ contain at most two points scheme and thus non-isomorphic points in $\pi(C)$ under a generic projection of $C$ into $\p^2$ are only nodes, whose set is defined by $\sqrt{K_1(I_C)}$. If $q^{'}=\pi(p^{'})$ is such a node then one knows $\delta_{p^{'}}=0$ and $\delta_{q^{'}}=1$ since $p^{'}\in C$ is a smooth point and $q^{'}\in\pi(C)$ is a nodal point. Hence we have
\begin{equation}\label{eq:301}
\deg(\bar R/\sqrt{K_1(I_C)})=\sum_{q \in \pi(C)} \delta_{q}-\sum_{p \in C} \delta_{p}.
\end{equation}
On the other hand, consider the short exact sequence:
$$
0 \rightarrow \mathcal{O}_{\pi(C)} \rightarrow \pi^{'}_{*}\mathcal{O}_{\widetilde{C}} \rightarrow \sum_{q \in \pi(C)} (\pi^{'}_{*}\mathcal{O}_{\widetilde{C}})_q/\mathcal{O}_{\pi(C),q} \rightarrow 0.
$$
Then we also obtain the following equation
\begin{equation*}\label{3}
\chi (\mathcal{O}_{\pi(C)}) - \chi (\pi_{*}^{'}(\mathcal{O}_{\widetilde{C}})) + \sum_{q \in \pi(C)} \delta_{q}=0,
\end{equation*}
which implies that
\begin{equation}\label{eq:302}
\sum_{q \in \pi(C)} \delta_{q}= \chi (\pi^{'}_{*}(\mathcal{O}_{\widetilde{C}}))-\chi (\mathcal{O}_{\pi(C)})
= \binom{d-1}{2} - \rho_{a}(\widetilde{C}).
\end{equation}

So, we have
$$
\begin{array}{lllllllllll}
\deg(\bar R/\sqrt{K_1(I_C)})&=\sum_{q \in \pi(C)} \delta_{q}-\sum_{p \in C} \delta_{p} &(\mbox{by equation (\ref{eq:301})})\\
                                           &=\binom{d-1}{2}-\rho_a (\widetilde{C})- \left(\rho_{a} (C) - \rho_{a}(\widetilde{C})\right)&(\mbox{by equation (\ref{eq:300}) and (\ref{eq:302})})\\
                                           &=\binom{d-1}{2}-\rho_a (C).
\end{array}
$$
We know that $\deg(\bar R/K_1(I_C))=\binom{d-1}{2}-\rho_a (C)$ by Lemma~\ref{lem3.1}. Thus we have
$$\deg(\bar R/\sqrt{K_1(I_C)}) =  \deg(\bar R/K_1(I_C)) .$$
Since $K_1(I_C)$ defines a zero-dimensional scheme, we have $\sqrt{K_1(I_C)} = K_1(I_C)^{\mathrm{sat}}$. Then we conclude that $K_1(I_C)$ is a radical ideal defining a set of points with degree $\binom{d-1}{2}-\rho_a (C)$ since $K_1(I_C)$ is already saturated(see \cite[Theorem 4.1]{A}).
\end{proof}

\begin{Cor}\label{rdeg}
Let $C$ be a non-degenerate integral curve of degree $d$ and
arithmetic genus $\rho_{a}(C)$ in $\p^{3}$. Assume that $\delta_p=1$ for every singular point $p \in C$.
Then $K_1(I_C)$ is a reduced ideal defining a set $Y_1(C)$ which consists of distinct $\binom{d-1}{2}-p_a(C)$ points.
\end{Cor}
\begin{proof}
It is enough to show that the condition $\delta_p=1$ implies $\dim\Tan_p(C)=2$ for every singular point $p\in C$.
Let $m_p\subset S=\mathcal{O}_{C,p}$ be a maximal ideal and $\widetilde{S}=(f_{*}\mathcal{O}_{\widetilde{C}})_p$.
Since $\delta_p =1$, it is easy to check that $ m_p = m_p \widetilde S$ (as sets) and $\mathrm{length}_S (\widetilde S / m_p \widetilde S) = 2$.
Therefore $f^{-1}(p)$ consists of at most two points.
First, in case $f^{-1}(p)$ consists of one point then, there is a unique maximal ideal  $\widetilde m_p=(t)$ such that
$\widetilde m_p \supsetneq m_p \widetilde S$ in the regular local ring $\widetilde{S}$ and $m_p\widetilde S=(t^2)$.
Therefore, $\mathrm{length}_S (\widetilde S/(m_p\widetilde S)^2)=\mathrm{length}_{S} (\widetilde S/(t^4))=4$.
Since we have the following exact sequences
\begin{equation}\label{cotan-seq}
0 \longrightarrow S/{m_p}^2 \longrightarrow  \widetilde S/(m_p \widetilde S)^2 \longrightarrow \widetilde S / S \longrightarrow 0
\end{equation}
\begin{equation*}
0 \longrightarrow m_p/{m_p}^2 \longrightarrow  S/{m_p}^2 \longrightarrow S /m_p \longrightarrow 0,
\end{equation*}
by the additivity of the length functions, we have $\mathrm{length}_S (S/{m_p}^2)=3$ and $\dim \mathrm{Tan}_p(C)=\dim_k(m_p/{m_p}^2)=2.$
Now let us assume that $f^{-1}(p)$ consists of distinct two points. Then $\widetilde S$ has precisely two maximal ideals $\widetilde m_1$
and $\widetilde m_2$.
As $m_p\widetilde S \subseteq \widetilde m_1 \cap \widetilde m_2 \varsubsetneq \widetilde m_1 ,\widetilde m_2$ and
$\dim_k(\widetilde S/m_p\widetilde S)=2$, we get $m_p\widetilde S = \widetilde m_1 \cap \widetilde m_2 = \widetilde m_1 \widetilde m_2$ and
the Chinese Remainder Theorem yields an isomorphism
$$\widetilde S/(m_p\widetilde S)^{2} \cong \widetilde S / \widetilde m_1^{2} \times \widetilde S / \widetilde m_2^{2}.$$
Since $\widetilde S_{\widetilde m_i}$ are regular local rings of dimension $1$, it can be checked that for $i = 1,2$
$$\mathrm{length}_S (\widetilde S / \widetilde m_i^{2})=\mathrm{length}_{S}(\widetilde S_{\widetilde m_i}/\widetilde m_i^{2}\widetilde S_{ \widetilde m_i})=2.$$
Thus we obtain $\mathrm{length}_S (\widetilde S/(m_p\widetilde S)^{2})=4$ and  consequently, by the sequences (\ref{cotan-seq}) again,
it is shown that $\dim \mathrm{Tan}_p (C)=2.$
\end{proof}
%

\begin{Thm}\label{max2}
Let $I_{C}$ be the defining ideal of an integral curve $C$ of degree $d$ in $\p^3$,
with $\dim\Tan_p(C)=2$ for every singular point $p \in C$, then
\begin{itemize}
\item[(a)] $M(I_{C}) = \max \{ d , 1+ \binom{d-1}{2}-\rho_a(C)  \}.$
\item[(b)] $M(I_C)$ can be obtained from one of the following two monomial generators
$$
x_1^d, \,\, x_0 x_2^{\binom{d-1}{2}-\rho_a(C)}.
$$
\end{itemize}
\end{Thm}
\begin{proof}
Note that by Theorem 3.5 in \cite{A},
$$
M(I_C)=\underset{k \geq 0}{\max} \{\reg(\Gin(K_k(I_C)))+k \}.$$
Let $s=\max \{ d, 1+ \binom{d-1}{2}-\rho_a(C) \}$.
Since $K_0(I_C)$ defines a plane curve $\pi(C)$ of degree $d$
and $K_1(I_C)$ defines a set of points of degree $\binom{d-1}{2}-\rho_a(C)$,
$$
\reg(\Gin(K_0(I_C)))=d
$$
and
$$
\reg(\Gin(K_1(I_C)))=\binom{d-1}{2}-\rho_a(C) ~~\mbox{(Theorem \ref{rdeg})}.
$$
This means that $M(I_C) \geq s.$

Conversely, to prove that $M(I_C) \leq s$ it suffices to show that
$$
\reg(\Gin(K_t(I_C))) \leq s-t ~\mbox{for all $t \geq 2$.}
$$
Let $\bar{R}_{t}=\bar{R}/K_{t}(I_{C})$ for each $t \geq 0$. We know
that $\bar{R}_{t}$ is an Artinian ring for $t \geq 2$ and from the
definition of regularity using the local cohomology, that
$\reg(K_{t}(I_{C}))=\min \{ m | H(\bar{R}_{t},m)=0 \}$. Now, we will
prove that if $m \geq s$ then $H(\bar{R}_{t},m-t)=0$, for all $t
\geq 2$. It is enough to show that for all
$m \geq s$
\begin{equation*}
H(R/I,m) = H(\bar{R}_{0},m) + H(\bar{R}_{1},m-1).
\end{equation*}
By the regularity bound,
\begin{equation}\label{glp}
H(R/I,m)=P_C(m) ~\mbox{if $m \geq s \geq d.$}
\end{equation}
Note that $Y_0(C)$ is a plane curve of degree $d$ in $\p^2$ and
$Y_1(C)$ is a reduced set of points of degree
$\binom{d-1}{2}-\rho_a(C).$\\
Thus if $m \geq s$ then $m\geq \reg ~Y_i(C), i=0,1$ and thus,
$$
H(\bar{R}_0, m)=P_{Y_0(C)}(m),
$$
$$
H(\bar{R}_1, m-1)=P_{Y_1(C)}(m-1)=\binom{d-1}{2}-\rho_a(C).
$$
Consequently, we have that if $m \geq s$ then
\begin{equation*}
\begin{split}
H(R/I,m) = & P_S(m)=P_{Y_0(C)}(m)+P_{Y_1(C)}(m-1) \\
         = & H(\bar{R}_0, m)+H(\bar{R}_1, m-1).
\end{split}
\end{equation*}
For a proof $(b)$, Since a generic projection of $C$ is a hypersurface of degree $d$ in $\p^2$,
we have that $\Gin(K_0 (I_C))=(x_1 ^d)$ by the Borel fixed property.
Furthermore we can consider all monomial generators of the form $x_0 \cdot h_j(x_1,x_2,x_3)$ in $\Gin(I_C)$.
Then, $\{ h_j(x_1,x_2,x_3) \}$ is a minimal generating set of $\Gin(K_1 (I_C))$ by Proposition~\ref{thm2}.
Recall that $K_1 (I_C)$ defines $\binom{d-1}{2}-\rho_a(C)$ distinct nodes in $\p^2$.
Thus $\Gin(K_1(I_C))$ should contain the monomial $x_2^{\binom{d-1}{2}-\rho_a(C)}$.
Therefore, $\Gin(I_C)$ contains monomials $x_1^d , x_0 x_2^{\binom{d-1}{2}-\rho_a(C)}$.
\end{proof}

\begin{remk}\label{remk:3.6}
Let $C \subset \p^r$, $r \geq 4$ with $\dim\Tan_p(C)=2$ for every singular point $p \in C$.
Consider the generic projection $\pi_{\Lambda}$ from a generic $(r-4)$-dimensional linear subvariety $\Lambda \subset \p^r$.
Since $\dim\Tan_p(C)=2$ for every singular point $p \in C$ we know that a generic projection $\pi_{\Lambda} : C \longrightarrow \p^3$ is an isomorphism and
$M(I_C) = M(I_{\pi_{\Lambda}(C)})$ by Remark~\ref{remk}(b). Thus we may assume that $I_{\pi_{\Lambda}(C)} \subset k[x_{r-3},\ldots , x_{r}]$ and
$M(I_C)$ can be obtained from one of the following two monomial generators
$$
x_{r-2}\,^d, \,\, x_{r-3} x_{r-1}\,^{\binom{d-1}{2}-\rho_a(C)}.
$$
\end{remk}

Therefore we get the following Corollary~\ref{maincor}.
\begin{Cor}\label{maincor}
Let $I_{C}$ be the defining ideal of an integral curve $C$ of degree $d$ in $\p^r$, $r \geq 4$
with $\dim\Tan_p(C)=2$ for every singular point $p \in C$, then
\begin{itemize}
\item[(a)] $M(I_{C}) = \max \{ d , 1+ \binom{d-1}{2}-\rho_a(C)  \}.$
\item[(b)] $M(I_C)$ can be obtained from one of the following two monomial generators
$$
x_{r-2}\,^d, \,\, x_{r-3} x_{r-1}\,^{\binom{d-1}{2}-\rho_a(C)}.
$$
\end{itemize}
\end{Cor}

\begin{prop}\label{mainthm}
Let $C$ be a non-degenerate integral curve of degree $d$ and arithmetic genus $p_a(C)$ in $\p^{r}$, $r \geq 3$,
with $\dim\Tan_p(C)=2$ for every singular point $p \in C$. Then
\[
M(I_{C})=
\begin{cases}
3  &\text{if $d=3$, i.e. $C$ is a rational curve of minimal degree;}\\
4  &\text{if $d=4$, i.e. $C$ is of next to minimal degree;}\\
1+\binom{d-1}{2}-\rho_{a}(C) &\text{for}\,\, d\ge 5.\\
\end{cases}
\]
\end{prop}
\begin{proof}
From Remark~\ref{remk:3.6}, we can reduce the case of an integral curve $C$ in $\p^3$.  
By Theorem~\ref{max2},
\begin{equation*}\label{9}
M(I_C)=\reg(\Gin_{\lex}(I_{C})) = \max \{d, 1+\binom{d-1}{2}-\rho_{a}(C) \}
\end{equation*}

Applying the genus bound in the Montreal lecture note of Eisenbud and Harris(1982)
to a non-degenerate integral curve $C\subset \p^3$, we get
\[
\rho_{a}(C) \leq  \pi(d,3) =
\begin{cases}
(\frac{d}{2}-1)^{2}  &\text{if $d$ is even;}\\
(\frac{d-1}{2})(\frac{d-3}{2})  &\text{if $d$ is odd.}
\end{cases}
\]
and for all $d \geq 5$, we have the following inequality:
\begin{equation}\label{10}
\rho_{a}(C) \leq \pi(d,3) \leq 1+\binom{d-1}{2} - d.
\end{equation}
Thus,
$$
d \leq 1+\binom{d-1}{2} - \rho_{a}(C)
$$
and by Theorem ~\ref{max2}, for $d \geq 5$,
$$M(I_{C})=1+\binom{d-1}{2}-\rho_{a}(C).$$
For special two cases of $d=3$ and $d=4$, it is very easy to compute $M(I_C)$. \\
If $d=3$ then $C$ is a rational normal curve and
$1+\binom{d-1}{2}-\rho_{a}(C)=2 < 3 = \deg(C)$.
Therefore, $M(I_C)=3$.
On the other hand, when $d=4$, we get the inequality $\rho_{a}(C)\leq \pi(4,3)= 1$.
Since $1+\binom{d-1}{2}-\rho_{a}(C)= 3 \,\,\text {or}\,\, 4$, we have $M(I_C)=4$.
\end{proof}

\begin{Ex}[\texttt{Singular}~\cite{DGPS}, {\it Macaulay 2}~\cite{GS}]
We revisit the example \cite[Example 4.3]{CS} introduced by A. Conca and J. Sidman.
$I_C=(x^3-yz^2,y^3-z^2t)$ defines an irreducible complete intersection curve $C$ of the arithmetic genus $\rho_a(C)=10$
in $\p^3$ with only one singular point $q=[0,0,0,1]$.
Note that this singular point is neither node nor ordinary cusp and $\delta_q = 10$.
We can compute the defining ideal of the normalization of a curve $C$ and delta invariant $\delta_q$ using \texttt {Singular}.
Furthermore, since $\dim \mathrm{Tan}_q (C)=3$, $\pi(q)$ is contained in the zero locus of $K_1 (I_C)$.
Thus we can not apply our results.
In fact, $\Gin(K_1 (I_C))$ is
\begin{align*}
(& y^4 ,y^3z^2 ,y^2z^5, yz^8, {\bf z^{15} },y^2z^4t,y^3zt^2,y^2z^3t^2,yz^7t^2,\\
 &y^3t^3,y^2z^2t^4,yz^6t^4,y^2zt^5,yz^5t^6,y^2t^7 ,yz^4t^8,yz^3t^{10} ).
\end{align*}
Therefore, $M(I_C)=1+M(K_1(I_C))=16$ which is not equal to $1+\binom{9-1}{2}-\rho_a(C)=1+28-10=19$.
\end{Ex}

\begin{Ex}[\texttt{Singular}~\cite{DGPS}, {\it Macaulay 2}~\cite{GS}]
Consider the ideal $I_C=( x^4-yz^3, y^2-zt) \subset k[x, y,z,t]$. This defines an irreducible complete intersection curve $C$ of $\rho_a(C)=10$ in $\p^3$ with one singular point $q=[0,0,0,1].$ The delta invariant $\delta_q$ is $9$ by \texttt{Singular}. Since $\dim \Tan_q(C)=2$, we can compute
by our formula,
$M(I_C)=1+\binom{8-1}{2}-9=13.$ In fact, $\Gin(I)$ is
$$
(x^2,xy^3 ,y^8, xy^2 z^2,xyz^5, {\bf xz^{12}} ,xy^2 zt^2,xyz^4 t^2,xy^2 t^4,xyz^3 t^4,xyz^2 t^6,xyzt^8,xyt^{10} )
$$
\end{Ex}

\begin{Ex}[\texttt{Singular}~\cite{DGPS}, {\it Macaulay 2}~\cite{GS}]
We consider the further example. Let $I_C=( x^3-yz^3, y^3-zt^2) \subset k[x, y,z,t]$ define an irreducible complete intersection curve $C$ in $\p^3$ with one singular point $q=[0,0,0,1].$ The dimension of a tangent space of $C$ at $q=[0,0,0,1]$ is $2$ and delta invariant $\delta_q = 6$. This singular curve has
$M(I_C)=19$ by computation of our formula. In particular, $\Gin(I_C)$ has a monomial generator ${\bf xz^{18}}$ having the maximal degree.
\end{Ex}

\begin{Ex}[{\it Macaulay 2}~\cite{GS}]
Let $C$ be a rational normal curve of degree 4 in $\p^4$ and $C_1$ be a projection curve in $\p^3$ with center $q \in \mathrm{Sec(C)} \setminus C$.
Then $C_1$ has one singular point as a node. Since the arithmetic genus of $C_1$ is $1$,
we know that
$$
d=4 \geq 3= 1+\binom{4-1}{2}-1
$$
Thus $M(I_{C_1})=4$ by our Theorem ~\ref{max2}. \\
In fact, we can compute the $\Gin(I_{C_1})$ with using {\it Macaulay 2}.
$$
\Gin(I_{C_1}) = (x_1^2,x_1x_2,{\bf x_2^4 },x_1x_3^2)
$$
\end{Ex}

\begin{Ex}[{\it Macaulay 2}~\cite{GS}]
Let $C$ be a rational normal curve in $\p^5$ and $C_1$ be a projection curve in $\p^4$ with center $q \in \mathrm{Sec(C)}\setminus C$.
Then $C_1$ has one singular point as a node. Consider a singular curve $C_2 $ in $\p^3$ as a generic projection of $C_1$.
In fact, $C_2$ is a singular curve of degree $5$, and the arithmetic genus $\rho_a(C_2)=1$ which is isomorphic to $C_1$.
Thus,
$$
M(I_{C_1})=M(I_{C_2})=1+\binom{d-1}{2}-\rho_a(C_2)=6.
$$
On the other hand, we can compute $\Gin(I_{C_2})$ using {\it Macaulay 2}.
$$
\Gin(I_{C_2}) = ( x_2^3,x_2^2x_3,x_2x_3^3,x_3^5,x_2^2x_4,x_2x_3x_4^2,{\bf x_2x_4^5}, x_2x_3x_4x_5,x_2x_3x_5^2 )
$$
\end{Ex}

\begin{Ex}[{\it Macaulay 2}~\cite{GS}]
Let $C$ be a rational normal curve in $\p^5$ and $C_1$ be a projection curve in $\p^4$ with center $q \in \mathrm{Sec(C)}\setminus C$.
Then $C_1$ has one singular point as a node. Let $q_1$ be a point in another secant line of $C$ and $\bar{q}_1 \in \p^4$ be an image of $q_1$
from the projection to $C_1$. Note that $\bar{q}_1$ is also a point of $\mathrm{Sec(C_1)} \setminus C_1$. If we project $C_1$ to $\p^3 $
from the  center $\bar{q}_1$
then we obtain the singular curve $C_2$ with the arithmetic genus $\rho_a(C_2)=2$ and ordinary two nodes.
Thus, by our formula,
$$
M(I_{C_2})=1+\binom{5-1}{2}-2=5.$$
Also, we can compute the generic initial of $I_{C_2}$ using {\it Macaulay 2}.
$$
\Gin(I_{C_2})=(x_2^2,x_2x_3^2,x_3^5,x_2x_3x_4, {\bf x_2x_4^4},x_2x_3x_5^2).
$$
\end{Ex}


\begin{remk}
Let $X$ be an irreducible reduced variety of codimension two. It is still open to compute or estimate $M(I_X)$
for $\dim (X)\ge 2$ (cf. \cite{AKS}). However, if $X$ is smooth or has a mild singularities, then it is expected that $M(I_X)$
is determined by the degree complexity of the double point locus under a generic projection.
Thus, by the induction on the dimension of hyperplane sections,
we expect asymptotically that
$$M(I_X) \thicksim 2(\frac{d}{2})^{2^n}.$$
\end{remk}

\bibliographystyle{amsalpha}

\end{document}